\documentclass[review]{elsarticle} 
\usepackage[usenames]{color}
\usepackage{lineno,hyperref}
\modulolinenumbers[5] 
\usepackage{geometry}
 \biboptions{numbers,sort&compress}
\usepackage{amsfonts}
\usepackage{amsmath, cases}
\usepackage{indentfirst}
\usepackage{graphicx}
\usepackage{latexsym,bm, amsthm}
\usepackage{epsfig}
\usepackage{latexsym}
\usepackage{ulem}
\usepackage{amssymb}\numberwithin{equation}{section}
\newtheorem{theorem}{T{\scriptsize HEOREM}}[section]
\newtheorem{lemma}[theorem]{L{\scriptsize  EMMA}}
\newtheorem{corollary}[theorem]{C{\scriptsize OROLLARY}}

\newtheorem{remark}{R{\scriptsize  EMARK}}[section]
\newtheorem{definition}{D{\scriptsize  EFINITION}}[section]

\newtheorem{example}{E{\scriptsize  XAMPLE}}[section]

\newcommand{\Drk}{{\rm Rank}}
\newcommand{\Ark}{{\rm ARank}}
\newcommand{\rk}{{\rm rank}}
\newcommand{\ind}{{\rm ind}}
\newcommand{\Dind}{{\rm Ind}}
\newcommand{\Aind}{{\rm AInd}}
\begin{document} 

\begin{frontmatter}
\title{Weak Dual Drazin Inverse and its Characterizations and Properties}

\author[1]{Hongxing Wang\corref{correspondingauthor}}
\ead{winghongxing0902@163.com}
\author[1]{Qiuli Ling}
\ead{l13211322247@163.com}
\author[2]{Tianhe Jiang}
\ead{Yeoljiang@163.com}
\author[3]{Shuangzhe Liu}
\ead{shuangzhe.liu@canberra.edu.au}
\cortext[correspondingauthor]{Corresponding author}
 \address[1]{School of Mathematical Sciences,
	Guangxi Key Laboratory of Hybrid Computation and IC Design Analysis,
	Guangxi Minzu University,
	Nanning, 530006,  China}

 \address[2]{Office of Computer and Mathematics, 
Department of Public Infrastructure, 
Guangxi Health Science College,
Nanning 530023, China}

\address[3]{Faculty of Science and Technology,
University of Canberra,  Canberra, ACT 2617, Australia} 

\begin{abstract}
The dual Drazin inverse is an important dual generalized inverse. In this paper, to extend it we introduce the weak dual Drazin inverse which is unique and exists for any square dual matrix. When  the dual Drazin inverse exists, it  coincides with  the weak dual Drazin inverse.  In addition, we   introduce the weak dual group inverse  and apply it to studying one type restricted dual matrix equation. 
\end{abstract}

\end{frontmatter}

\linenumbers

\section{Introduction}\label{sec1}

As a powerful and convenient mathematical tool, the dual numbers and their algebras
  are  widely used in mechanical engineering fields,
 such as kinematic analysis \cite{yundongxuefenxi--1, yundongxuefenxi--2, Wei2024siammaa-A,Wei2024coam},
 robotics \cite{jiqire--1,jiqire--2,Chen2024jota-A,Qi2023camc}.
 In addition,
 they  are also found  in solving linear equations
 in screw motion \cite{luoxuanyundong} and rigid body motion \cite{gangtiyundong}.

In this paper,  we consider
 to form a dual number $\widehat{a}=a+\varepsilon a_{0}$
combining the standard part $a$ and the dual part $\varepsilon a_{0}$,
 where $a$ and $a_{0}$ are both real numbers. It  is noted that
 the dual unit $\varepsilon$ meets the requirements
  $\varepsilon^{2}=0$,
$0\varepsilon=\varepsilon0=0$,
$1\varepsilon=\varepsilon1=\varepsilon$, but $\varepsilon\neq 0$.
 Let $\mathbb{D}$ be the set of dual numbers,
and $\mathbb{D}^{m\times n}$
 be the set of $m\times n$ dual matrices.
Since every entry of a dual matrix is   a  dual number,
the dual matrix $\widehat{M}\in\mathbb{D}^{m\times n}$ can be expressed as
$\widehat M=M+\varepsilon M_0$,
where $M, M_{0}\in\mathbb{R}^{m\times n}$.
 The symbol $I_{n}$ stands for the $n\times n$ identity matrix.
Let $M \in \mathbb{R}^{n \times n}$,
if there exists a positive integer $k$
  such that  $\rk(M^{k+1})=\rk(M^{k})$,
then $k$ is called the index of $M$,
 and is denoted by ${\ind}(M) = k$.
Let two dual matrices $\widehat{M}\in\mathbb{D}^{n\times n}$  and   $\widehat{L}\in\mathbb{D}^{n\times n}$, if $\widehat{M}\widehat{L}=\widehat{L}\widehat{M}=I_{n}$, then $\widehat{M}$ is invertible and $\widehat{L}$ is the dual inverse matrix of $\widehat{M}$, denoted as $\widehat{L}=\widehat{M}^{-1}$.
However, the inverse of the dual matrix does not always exist,
and can be used to solve relevant practical problems only when it exists.
Such problems are as dual linear systems  in the field of mechanical engineering \cite{yundongxuefenxi--2}.
Consequently, to address this issue, the dual generalized inverse is presented.

Recently,
the dual generalized inverse matrices of dual matrix have attracted a lot of attention, especially the dual Moore-Penrose generalized inverse.
 Given a $m\times n$ dual matrix $\widehat M$,
if there exists a dual matrix $\widehat{X}\in\mathbb{D}^{n\times m}$ satisfying the  Penrose conditions:
\begin{align}
\label{DMPGI-Def}
\widehat{M}\widehat{X}\widehat{M}=\widehat{M},\
\widehat{X}\widehat{M}\widehat{X}=\widehat{X},\
\left(\widehat{M}\widehat{X}\right)^{T}=\widehat{M}\widehat{X},\
\left(\widehat{X}\widehat{M}\right)^{T}=\widehat{X}\widehat{M},
\end{align}
then $\widehat{X}$ is called the   {dual Moore-Penrose inverse} (DMPI)
of $\widehat{M}$, and denoted as $\widehat{M}^{\dag}$.
Wang \cite{Wang--DMPGI} gets two necessary and sufficient conditions for a dual matrix to have the DMPI.

In \cite{zhong--dual group inverse},
Zhong and Zhang introduce the dual group inverse
and consider the existence, computation and applications of the inverse.
 Given a $n\times n$ dual matrix $\widehat M=M+\varepsilon M_0$  with $\ind(M)=1$.
If there exists a dual matrix $\widehat{X}\in\mathbb{D}^{n\times n}$ satisfying
\begin{align}
\label{DGI-Def}
\widehat{M}\widehat{X}\widehat{M}=\widehat{M},\
\widehat{X}\widehat{M}\widehat{X}=\widehat{X},\
\widehat{M}\widehat{X}=\widehat{X}\widehat{M},
\end{align}
then $\widehat{X}$ is called the   {dual group   inverse} (DGI) 
of $\widehat{M}$,
and  denoted as $\widehat{X}=\widehat{M}^{\#}$.
The DGI is applied to the least-squares problem of  dual linear systems.
In \cite{Wang and Gao},
Wang and Gao introduce  the
dual index and the dual core  inverse.
Given a $n\times n$ dual matrix $\widehat M=M+\varepsilon M_0$ with $\ind(M)=1$,
if there exists a dual matrix $\widehat{X}\in\mathbb{D}^{n\times n}$ satisfying
\begin{align}
\label{DCI-Def}
\widehat{M}\widehat{X}\widehat{M}=\widehat{M},\
\widehat{X} \widehat{M}\widehat{X} =\widehat{X},\
\left(\widehat{M}\widehat{X}\right)^{T}=\widehat{M}\widehat{X},
\end{align}
then $\widehat{X}$ is called the    {dual  core   inverse} (DCI) of $\widehat{M}$,
and denoted as  $\widehat{X}=\widehat{M}^{\tiny\textcircled{\#}}$.
%
Wang and Gao \cite{Wang and Gao}
  prove  that the DCI of $\widehat M$ exists if and only if  the index of $\widehat M$ is 1,
and provide  some properties and characterizations of the inverse.
Furthermore, Zhong and Zhang \cite{zhong--dual Drazin inverse}
introduce  the dual Drazin   inverse.
Given a $n\times n$ dual matrix $\widehat M=M+\varepsilon M_0$, ${\mbox{\ind}}\left(M\right)=k$,
if there exists a dual matrix $\widehat{X}\in\mathbb{D}^{n\times n}$ satisfying
\begin{align}
\label{DDI-Def}
\widehat{M}\widehat{X}\widehat{M}^{k}=\widehat{M}^{k},\
\widehat{X}\widehat{M}\widehat{X}=\widehat{X},\
\widehat{M}\widehat{X}=\widehat{X}\widehat{M},
\end{align}
then $\widehat{X}$ is called the  {dual Drazin inverse} (DDI) of $\widehat{M}$,
and denoted as  $\widehat{X}=\widehat{M}^{D}$.
When it exists,
Zhong and Zhang \cite{zhong--dual Drazin inverse}
 prove that the DDI  is unique,
    provide  some sufficient and necessary conditions for its existence
  and give a compact formula for the DDI.

It should be pointed out that
the  above-mentioned  dual generalized inverses do not exist for every dual matrices.
Therefore, further work is needed.
In \cite{qi20230804,Li and Wang},
 the authors introduce the weak dual generalized inverse
that every dual matrix has the dual generalized inverse.
 Given a $m\times n$ dual matrix $\widehat M$,
 there exists a unique dual matrix  $\widehat{X}\in\mathbb{D}^{n\times m}$
   satisfying
\begin{align}
\label{g-DMPGI-Def}
\widehat{M}^{T}\widehat{M}\widehat{X}\widehat{M}\widehat{M}^{T}=
\widehat{M}^{T}\widehat{M}\widehat{M}^{T},\
\widehat{X}\widehat{M}\widehat{X}=\widehat{X},\
\left(\widehat{M}\widehat{X}\right)^{T}=\widehat{M}\widehat{X},\
\left(\widehat{X}\widehat{M}\right)^{T}=\widehat{X}\widehat{M},
\end{align}
then $\widehat{X}$ is called the  {weak dual generalized inverse} (WDGI) of $\widehat{M}$,
and denoted as  $\widehat{X}= \widehat{M}^{ \dag}_{\rm W}$.
It is easy to check that
$\widehat{M}^{\dag}_{\rm W}
=\left( \widehat{M}^{T}\widehat{M} \right)^\dag\widehat{M}^{T}
=\widehat{M}^{T}\left(\widehat{M}\widehat{M}^{T} \right)^\dag$.
The WDGI is  a generalization of DMPI and
can be used to solve more general dual linear system problems \cite{qi20230804,Li and Wang}.

  In this paper,
motivated by the works proposed above,
we consider extending  the DDI
so that the new inverse of any square dual matrix exists.
Under the condition that the DDI exists, the new inverse coincides with the DDI.
We also discuss its properties and characterizations.
Furthermore, we introduce a generalized DGI,
provide its characterizations,
and apply it to discuss problems such as constrained dual matrix equations.

\section{Preliminaries}
\label{Section-2-Preliminaries}
In this section,
 we present some preliminary results.
The first is the core-EP decomposition, which is the most powerful tool for studying the Drazin inverse.

\begin{theorem}
[{\cite{Generalized Inverses book}}]
Let $M\in\mathbb{R}^{n\times n}$
with ${\ind}\left( M \right)=k$ and
${\rm rank}\left( M^k \right)=r$.
Then there exists an invertible matrix
$  P  $
such that 
\begin{align}
\label{R-C-N-Decomp}
 M=P\begin{bmatrix}
        C&O  \\
        O&N         \end{bmatrix}P^{-1},
\end{align}
where  $C\in\mathbb{R}^{r\times r}$ is nonsingular and $N$ is nilpotent.
\end{theorem}

Based on the above decomposition,
 a clear and concise characterization of the Drazin inverse is given.

\begin{lemma}
[\cite{Generalized Inverses book}]
\label{C-N-Decomposion-Th}
Let the core-nilpotent decomposition of $M\in\mathbb{R}^{n\times n}$  be given as in (\ref{R-C-N-Decomp}),
then
\begin{align}
\label{R-C-N-Decomp-Drazin}
M^D=P\begin{bmatrix}
        C^{-1}&O  \\
        O&O        \end{bmatrix}P^{-1}.
\end{align}
\end{lemma}

By applying (\ref{R-C-N-Decomp}) and (\ref{R-C-N-Decomp-Drazin}),
 Zhong et al. \cite{zhong--dual group inverse,zhong--dual Drazin inverse}
 get the characterizations of the DGI and DDI, respectively.

\begin{lemma}
[\cite{zhong--dual group inverse}]
\label{le-DGI-CGGI}
Let   $\widehat M=M+\varepsilon M_0  \in\mathbb{D}^{n\times n} $,
 ${\mbox{\ind}}\left({M}\right)=1$ and $\left(I_n-MM^{\dag}\right)M_0\left(I_n-M^{\dag}M\right)=O$.
Then the DGI $\widehat{M}^{\#}$ of $\widehat{M}$ exists
and
\begin{align*}
\widehat{M}^{\#}=&M^{\#}
\\&+\varepsilon
\left({\left(M^{\#}\right)}^2M_0\left(I_n-MM^{\#}\right)
+\left(I_n-MM^{\#}\right)M_0{\left(M^{\#}\right)}^2-M^{\#}M_0M^{\#}\right).
\end{align*}
\end{lemma}

 \begin{lemma}
[\cite{zhong--dual Drazin inverse}]
\label{le-1.1}
Let  $\widehat M=M+\varepsilon M_0 \in\mathbb{D}^{n\times n} $ with
${\mbox{\ind}}\left(M\right)=k$.
Then $\widehat M^{k}=M^{k}+\varepsilon K_0 $, in which $K_0=\overset{k}{\underset{i=1}\sum}M^{k-i}M_0M^{i-1}$.
Furthermore, if $\left(I_n-MM^{D}\right)K_0\left(I_n-MM^{D}\right)=O$,
the DDI $\widehat M^{D}$ of $\widehat{M}$ exists and
$$\widehat M^{D}=M^{D}+\varepsilon S,$$
where
\begin{align}
\nonumber
S&={\left(M^D\right)}^2
\left(\overset{k-1}{\underset{i=0}\sum}{\left(M^D\right)}^iM_0M^i\right)
\left(I_n-MM^D\right)
\\
\nonumber
&\qquad\qquad +
\left(I_n-MM^D\right)
\left(\overset{k-1}{\underset{i=0}\sum}M^iM_0{\left(M^D\right)}^i\right)
{\left(M^D\right)}^2-M^DM_0M^D.
\end{align}
\end{lemma}

Let $ \widehat M=M+\varepsilon M_0\in\mathbb{D}^{n\times n}$.
Denote
the  appreciable rank of $ \widehat M$ as
$\Ark\left( \widehat M \right)=\rk (M)$,
 the rank of $ \widehat M$ as
$\Drk\left( \widehat M \right)
=
\rk\left[\begin{matrix}
 M_0  &     M           \\
 M      &    O
\end{matrix}\right]-\rk (M) $,
the appreciable index of $\widehat M $  as
$ {\Aind} \left(\widehat M\right)={\ind} \left(  M\right)$,
 and
 the   index of $\widehat M $ as
 $ {\Dind} \left(\widehat M\right)$,
which is the smallest positive integer $t$
$\left(k \leq t \leq 2k,\ k={\ind} \left(  M\right)\right)$
such that
${\rm ARank}\left(\widehat M^{t}\right)
= {\rm Rank}\left(\widehat M^{t}\right)$, \cite{Wang2023LaaSub}.

\begin{lemma}[\cite{Wang2023LaaSub}]
\label{dual-C-N}
Let
$\widehat  M = M + \varepsilon M_0 \in\mathbb{D}^{n\times n}$
with
${\mbox{\Dind}}\left(\widehat{M}\right)=t$,
${\Drk}\left( \widehat{M}^t \right)=r$.
Then there exists an invertible dual matrix
$\widehat  P  \in\mathbb{D}^{n\times n}$
such that 
\begin{align}
\label{C-N-D}
\widehat M =
\widehat P
\begin{bmatrix}
 \widehat C &       O     \\
        O   & \widehat N
\end{bmatrix}\widehat P^{-1},
\end{align}
where  $\widehat{C} \in \mathbb{D}^{r \times r}$ is  invertible,
and
$\widehat{N}$ is nilpotent, ie., $\widehat{N}^t=O$.

Furthermore,
when
${\mbox{\Dind}}\left(\widehat{M}\right)= {\Aind}\left( \widehat{M} \right)$,
the DDI $\widehat M^{D}$ of $\widehat{M}$ exists and
 \begin{align}
\label{C-N-D-DGI}
\widehat M^D =
\widehat P
\begin{bmatrix}
 \widehat C ^{-1}&       O     \\
        O   &O
\end{bmatrix}\widehat P^{-1}.
\end{align}
\end{lemma}

\section{Weak dual Drazin inverse}
%
In this section,
   we introduce   a generalized DDI.
   This inverse   exists for any square dual matrix.
 \begin{theorem}
\label{20220312-Th-1}
Let $\widehat{M}=M+\varepsilon M_0\in\mathbb{D}^{n\times n}$
with
${\mbox{\Dind}}\left(\widehat{M}\right)=t$.
Denote $K=\overset{t}{\underset{i=1}\sum}M^{t-i}M_0M^{i-1}$.
Then
 $\widehat M^{t}=M^{t}+\varepsilon K$
and
 there exists  a $n\times n$ dual matrix $\widehat X=X+\varepsilon S$ satisfying
\begin{align}
\label{GDGI-def}
\widehat M\widehat X\widehat M^{t}=\widehat M^{t}, \
\widehat X\widehat M\widehat X=\widehat X, \
\widehat M\widehat X=\widehat X\widehat M
\end{align}
if and only if
\begin{align}
\label{Th3.1-eq1}
X&=M^D,
\\
\label{Th3.1-eq2}
K&=MM^DK+MSM^{t}+M_0M^DM^{t},
\\
\label{Th3.1-eq3}
S&=M^DMS+M^DM_0M^D+SMM^D,
\\
\label{Th3.1-eq4}
MS+M_0M^D&=SM+M^DM_0.
\end{align}
\end{theorem}

\begin{proof}
Let $\widehat{M}=M+\varepsilon M_0\in\mathbb{D}^{n\times n}$
with
${\mbox{\Aind}}\left(\widehat {M}\right)=k$
and
${\mbox{\Dind}}\left(\widehat{M}\right)=t$,
and   $\widehat X=X+\varepsilon S\in\mathbb{D}^{n\times n}$.
Denote $K=\overset{t}{\underset{i=1}\sum}M^{t-i}M_0M^{i-1}$.
It is easy to check that
 $\widehat M^{t}=M^{t}+\varepsilon K$.
Substituting
 $\widehat X$ into
$\widehat M\widehat X\widehat M^{t}=\widehat M^{t}$,
$\widehat X\widehat M\widehat X=\widehat X$ and
$\widehat M\widehat X=\widehat X\widehat M$,
we get
\begin{align}
\nonumber
\left(M+\varepsilon M_0\right)
\left(X+\varepsilon S\right)(M^{t}+\varepsilon K)
&=MXM^{t}+\varepsilon\left(MXK+MSM^{t}+M_0SM^{t}\right)
\\
\nonumber
&
=M^{t}+\varepsilon K,
\\
\nonumber
\left(X+\varepsilon S\right)\left(M+\varepsilon M_0\right)(X+\varepsilon S)
&
=XMX+\varepsilon\left(XMS+XM_0X+SMX\right)
=X+\varepsilon S,
\\
\nonumber
MX+\varepsilon\left(MS+M_0X\right)
&=XM+\varepsilon\left(XM_0+SM\right).
\end{align}
Therefore,
$MXM^{t} =M^{t}$, $ X MX=X$, $ MX=XM$,
$K =MXK+MSM^{t}+M_0XM^{t}$,
$S =XMS+XM_0X+SMX$
and
$MS+M_0X =SM+XM_0$.
Since ${\mbox{\Aind}}\left(\widehat {M}\right)=k$
and
${\mbox{\Dind}}\left(\widehat{M}\right)=t$, $t\geq k$,
it follows from  $MXM^{t} =M^{t}$, $ X MX=X$ and $ MX=XM$ that
$X=M^D$.
So,
 we get (\ref{Th3.1-eq1}), (\ref{Th3.1-eq2}), (\ref{Th3.1-eq3}) and (\ref{Th3.1-eq4}).
\end{proof}

\begin{theorem}
\label{Th3.2}
Let $\widehat M=M+\varepsilon M_0\in\mathbb{D}^{n\times n}$ with
${ {\Dind}}\left(\widehat{M}\right)=t$.
Then  the solution to (\ref{GDGI-def}) is unique and
\begin{align}
\label{GDDI -def-Solution-1}
\widehat X=M^D+\varepsilon S
\end{align}
where
\begin{align}
\nonumber
S&= {\left(M^D\right)}^2
\left(\overset{t-1}{\underset{i=0}\sum}{\left(M^D\right)}^iM_0M^i\right)
\left(I_n-MM^D\right)
\\
\label{20231016-Eq-8}
&\qquad\qquad
+\left(I_n-MM^D\right)
\left(\overset{t-1}{\underset{i=0}\sum}M^iM_0{\left(M^D\right)}^i\right)
{\left(M^D\right)}^2
 -M^DM_0M^D.
\end{align}
\end{theorem}

\begin{proof}
Let $\widehat{M}=M+\varepsilon M_0\in\mathbb{D}^{n\times n}$,
${\mbox{\Aind}}\left(\widehat {M}\right)=k$,
${\mbox{\Dind}}\left(\widehat{M}\right)=t$,
${\rm rank}\left( M^k \right)=r$,
 the core-nilpotent decomposition of $M\in\mathbb{R}^{n\times n}$  be given as in (\ref{R-C-N-Decomp})
and  $\widehat X=M^D+\varepsilon S$ be the solution to (\ref{GDGI-def}).
And let $P^{-1}M_0 P$ and $P^{-1} S P$ be partitioned as
\begin{align}
\label{20231016-Eq-1}
P^{-1}M_0 P=\begin{bmatrix}
          M_1&M_2  \\
          M_3&M_4  \end{bmatrix},\
         P^{-1} S P=
          \begin{bmatrix}
          S_1&S_2  \\
          S_3&S_4
          \end{bmatrix},
\end{align}
where $P$ is given as in (\ref{R-C-N-Decomp}),
 $M_1\in\mathbb{R}^{r\times r}$ and $S_1\in\mathbb{R}^{r\times r}$.
 Applying (\ref{R-C-N-Decomp}) and (\ref{20231016-Eq-1}),
 it is easy to check that
  $\widehat{M}^t=M^t+\varepsilon K$, in which
\begin{align}
\label{20231016-Eq-2}
K=\overset{t}{\underset{i=1}\sum}M^{t-i}M_0M^{i-1}=P\begin{bmatrix}
\overset{t}{\underset{i=1}\sum}C^{t-i}M_1C^{i-1}&
\overset{t-1}{\underset{i=0}\sum}C^{t-1-i}M_2N^i \\
\overset{t-1}{\underset{i=0}\sum}N^iM_3C^{t-1-i}&O
\end{bmatrix}P^{-1}.
\end{align}

Denote
\begin{align}
\label{20231016-Eq-3}
\widehat G=P \begin{bmatrix}
C^{-1}&O  \\
     O&O   \end{bmatrix} P^{-1}+\varepsilon
P \begin{bmatrix}
     -C^{-1}M_1C^{-1}    & \overset{t-1}{\underset{i=0}\sum}C^{-i-2}M_2N^i \\
     \overset{t-1}{\underset{i=0}\sum}N^iM_3C^{-i-2} &    O
      \end{bmatrix} P^{-1}.
\end{align}

Applying (\ref{R-C-N-Decomp}),  (\ref{20231016-Eq-1}),  (\ref{20231016-Eq-2}) and
  (\ref{20231016-Eq-3}),
we get
\begin{align}
\nonumber
\widehat M\widehat G
&=P\begin{bmatrix}
     I_r&O  \\
     O&O\end{bmatrix}P^{-1}
+\varepsilon P
 \begin{bmatrix}
     O&\overset{t-1}{\underset{i=0}\sum}C^{-i-1}M_2N^i    \\
     \overset{t-1}{\underset{i=0}\sum}N^i M_3C^{-i-1}&O
     \end{bmatrix}
     P^{-1},
\\
\nonumber
\widehat G\widehat M
&=P\begin{bmatrix}
     I_r&O  \\
     O&O\end{bmatrix}P^{-1}
+\varepsilon P \begin{bmatrix}
     O                                 &\overset{t-1}{\underset{i=0}\sum}C^{-i-1}M_2N^i    \\
     \overset{t-1}{\underset{i=0}\sum}N^i M_3C^{-i-1}   &O                              \end{bmatrix} P^{-1},
\\
\nonumber
\widehat M\widehat G \widehat M^{t}
&=\left(P\begin{bmatrix}
     I_r&O  \\
     O&O\end{bmatrix}P^{-1}
+\varepsilon P \begin{bmatrix}
     O&\overset{t-1}{\underset{i=0}\sum}C^{-i-1}M_2N^i    \\
     \overset{t-1}{\underset{i=0}\sum}N^i M_3C^{-i-1}&O                     \end{bmatrix} P^{-1}\right)
\\
\nonumber
&
\times\left(P\begin{array}{c}\begin{bmatrix}
         C^{t}&O  \\
              O&O         \end{bmatrix}\end{array}P^{-1}
+\varepsilon P\begin{array}{c}\begin{bmatrix}
     \overset{t}{\underset{i=1}\sum}C^{t-i}M_1C^{i-1}   &     \overset{t-1}{\underset{i=0}\sum}C^{t-1-i}M_2N^i \\
\overset{t-1}{\underset{i=0}\sum}N^iM_3C^{t-1-i}&      O                 \end{bmatrix}\end{array}P^{-1}\right)
\\
\nonumber
&=P\begin{bmatrix}
        C^{t}&O  \\
        O&O   \end{bmatrix}P^{-1}
+\varepsilon P\begin{bmatrix}
\overset{t}{\underset{i=1}\sum}C^{t-i}M_1C^{i-1}&
\overset{t-1}{\underset{i=0}\sum}C^{t-1-i}M_2N^i    \\
\overset{t-1}{\underset{i=0}\sum}N^iM_3C^{t-1-i}&O                           \end{bmatrix}P^{-1}
\\
\nonumber
&=\widehat M^{t},
\end{align}
\begin{align}
 \nonumber
\widehat G\widehat M\widehat G
&
=\left(\widehat G\widehat M\right) \widehat G
\\
&
=P \begin{bmatrix}
C^{-1}&O  \\
     O&O   \end{bmatrix} P^{-1}+\varepsilon
P \begin{bmatrix}
     -C^{-1}A_1C^{-1}               &   \overset{t-1}{\underset{i=0}\sum}C^{-i-2}M_2N^i \\
     \overset{t-1}{\underset{i=0}\sum}N^iM_3C^{-i-2} &    O               \end{bmatrix} P^{-1}
\nonumber
 =\widehat G.
\end{align}
It follows 
  that  $ \widehat G $  is  the solution to (\ref{GDGI-def}).
Furthermore,
since
\begin{align*}
 P\begin{array}{c}\begin{bmatrix}
     -C^{-1}M_1C^{-1}  &   O \\
     O &    O               \end{bmatrix}\end{array}P^{-1}
&=M^DM_0M^D ,
\\
P\begin{array}{c}\begin{bmatrix}
     O               &   \overset{t-1}{\underset{i=0}\sum}C^{-i-2}M_2N^i \\
     O &    O               \end{bmatrix}\end{array}P^{-1}
&={\left(M^D\right)}^2
\left(\overset{t-1}{\underset{i=0}\sum}{\left(M^D\right)}^iM_0M^i\right)
\left(I_n-MM^D\right) ,
\\
P\begin{array}{c}\begin{bmatrix}
     O               &   O \\
     \overset{t-1}{\underset{i=0}\sum}N^iM_3C^{-i-2} &    O               \end{bmatrix}\end{array}P^{-1}
&=\left(I_n-MM^D\right)
\left(\overset{t-1}{\underset{i=0}\sum}M^iM_0{\left(M^D\right)}^i\right)
{\left(M^D\right)}^2,
\end{align*}
then we get (\ref{GDDI -def-Solution-1}) and   (\ref{20231016-Eq-8}).

Next, we prove the uniqueness of  $ \widehat G $.

Let
 $\widehat{X}_{1}=M^{D}+\varepsilon S_{1}$
and
$\widehat{X}_{2}=M^{D}+\varepsilon S_{2}$
be two solutions to (\ref{GDGI-def}).
Substituting  $\widehat{X}_{1}$ and  $\widehat{X}_{2}$ into
(\ref{Th3.1-eq2}), (\ref{Th3.1-eq3}) and (\ref{Th3.1-eq4})
gives
\begin{align*}
\begin{array}{ll}
K=M M^{D} K+M S_{1} M^{t}+M_{0} M^{D} M^{t}, & \ K =M M^{D} K+M S_{2} M^{t}+M_{0} M^{D} M^{t},
\\
S_{1} =M^{D} M S_{1}+M^{D} M_{0} M^{D}+S_{1} M M^{D}, & \ S_{2} =M^{D} M S_{2}+M^{D} M_{0} M^{D}+S_{2} M M^{D},
\\
M S_{1}+M_{0} M^{D}=S_{1} M+M^{D} M_{0}, & \ M S_{2}+M_{0} M^{D}=S_{2} M+M^{D} M_{0}.
\end{array}
\end{align*}
Then we get
\begin{align}
\label{Th3.3-eq1}
M\left(S_{1}-S_{2}\right) M^{t}&=O,
\\
\label{Th3.3-eq2}
S_{1}-S_{2}&=M^{D} M\left(S_{1}-S_{2}\right)+\left(S_{1}-S_{2}\right) M M^{D},
\\
\label{Th3.3-eq3}
M\left(S_{1}-S_{2}\right)&=\left(S_{1}-S_{2}\right) M.
\end{align}
Applying   (\ref{Th3.3-eq3})  to  (\ref{Th3.3-eq1}), we get that
$ \left(S_{1}-S_{2}\right) M^{t+1}=O$
  and
$ M^{t+1}\left(S_{1}-S_{2}\right)=O$.
It follows from  $MM^D=M^DM=\left(MM^D\right)^{t+1}$
that
$\left(S_{1}-S_{2}\right) M M^{D}=M M^{D}\left(S_{1}-S_{2}\right) =O$.
Furthermore, substituting
$M^{D}\left(S_{1}-S_{2}\right)=O$
and
$\left(S_{1}-S_{2}\right) M M^{D}=O$
 into (\ref{Th3.3-eq2})
gives
$S_{1}-S_{2}=O$.
Therefore,  the solution to (\ref{GDGI-def}) is unique.
\end{proof}

\begin{definition}
\label{GDDI-Def}
 Let $\widehat{M} \in \mathbb{D}^{n \times n}$.
 Then the unique solution to  (\ref{GDGI-def}) is called the weak DDI (WDDI) of $\widehat{M}$,
 and denoted as $\widehat{M}_{\rm W}^{D}$.
\end{definition}

The following example indicates that the WDDI of $\widehat M$ always exists,
even if the DDI of $\widehat M$ does not exist.

\begin{example}
\label{Ex-3-1}
Let
$$\widehat M=M+\varepsilon M_0
=
 \begin{bmatrix}
        1&1&0&0  \\
        0&0&1&1   \\
        0&1&0&0   \\
        1&0&0&0      \end{bmatrix}
+\varepsilon \begin{bmatrix}
        1&0&1&0  \\
        0&1&0&1   \\
        1&0&0&1   \\
        0&1&1&0     \end{bmatrix}.$$
Then  ${\Aind} \left(\widehat{M}\right)=k=2 $
and
${\Dind} \left(\widehat{M}\right)=t=4 $ and
\begin{align*}
M^D&=\begin{bmatrix}
        1&1&-1&-1  \\
        -1&-1&2&2   \\
        2&2&-3&-3   \\
        -1&-1&2&2       \end{bmatrix},\
M^DM_0M^D=\begin{bmatrix}
        0&0&0&0  \\
        1&1&0&0   \\
        -1&-1&0&0   \\
        1&1&0&0     \end{bmatrix},
\\
        K_0
 &=\overset{2}{\underset{i=1}\sum}M^{2-i}M_0M^{i-1}
=\begin{bmatrix}
        2&3&1&1  \\
        2&1&2&2   \\
        2&2&0&1   \\
        1&1&2&1      \end{bmatrix}, 
\\
\left(I_4-MM^D\right)&K_0\left(I_4-MM^D\right)
=\begin{bmatrix}
        -1&0&1&1  \\
        1&0&-1&-1   \\
        -1&-1&1&2   \\
        1&1&-1&-2       \end{bmatrix}
\neq O.
\end{align*}
It follows  from Lemma \ref{le-1.1}
   that the DDI of $\widehat{M}$ does not exist.

Applying (\ref{20231016-Eq-8}) gives
 \begin{align*}
S={\left(M^D\right)}^2&
\left(\overset{3}{\underset{i=0}\sum}{\left(M^D\right)}^iM_0M^i\right)
\left(I_4-MM^D\right)
\\
&
+\left(I_4-MM^D\right)
\left(\overset{3}{\underset{i=0}\sum}M^iM_0{\left(M^D\right)}^i\right)
{\left(M^D\right)}^2-M^DM_0M^D
\\
&=\begin{bmatrix}
        -4&-4&6&6  \\
        3&3&-6&-6   \\
        -4&-4&8&8   \\
        4&4&-8&-8       \end{bmatrix}.
\end{align*}
Then
 by applying (\ref{GDDI -def-Solution-1}) we have
$$\widehat{M}_{\rm W}^{D}=M^{D}+\varepsilon S=\begin{bmatrix}
        1&1&-1&-1  \\
        -1&-1&2&2   \\
        2&2&-3&-3   \\
        -1&-1&2&2      \end{bmatrix}
+\varepsilon \begin{bmatrix}
        -4&-4&6&6  \\
        3&3&-6&-6   \\
        -4&-4&8&8   \\
        4&4&-8&-8       \end{bmatrix}.$$
\end{example}

\begin{remark}
\label{Remark-GDDI-1}
Let $\widehat M=M+\varepsilon M_0\in\mathbb{D}^{n\times n}$ with
${ {\Aind}}\left(\widehat {M}\right)=k$,
and
${ {\Dind}}\left(\widehat{M}\right)=t$.
From Lemma \ref{le-1.1},
we see that the DDI of $\widehat{M}$ exists if and only if the DMPI of $\widehat{M}^k$ exists.
Therefore, applying Lemma \ref{le-DGI-CGGI} we get that the DCI of $\widehat{M}^k$ exists,
 that is, the index of $\widehat{M}^k$ is 1.
It follows that  $\Dind \left(\widehat{M}^k\right)=\Aind \left(\widehat{M}^k\right)$.
So the index of $\widehat{M}$ is $k$.
Furthermore, applying (\ref{DDI-Def}) and (\ref{GDGI-def}),
we get that  $\widehat{M}^{D}=\widehat{M}_{\rm W}^{D}$
when the DDI of $\widehat{M}$ exists.
\end{remark}

\begin{remark}
\label{Remark-GDDI-2}
Let $\widehat M=M+\varepsilon M_0\in\mathbb{D}^{n\times n}$ with
${ {\Aind}}\left(\widehat {M}\right)=k$,
and
${ {\Dind}}\left(\widehat{M}\right)=t$.
From
Remark \ref{Remark-GDDI-1},
\cite[Theorem 2.2]{zhong--dual Drazin inverse}
and
\cite[Theorem 2.6. and Theorem 3.3]{Wang and Gao},
  the following conditions are equivalent:
\begin{enumerate}
  \item[{\bf  (1)}]
  The DDI  of  $\widehat M$ exists
  and $\widehat{M}^{D}=\widehat{M}_{\rm W}^{D}$

  \item[{\bf  (2)}]
  The DMPI (or DCI or DGI)   of  ${\widehat M} ^{k}$ exists;

  \item[{\bf  (3)}]
 The index of $\widehat{M}^k$ is 1;

  \item[{\bf  (4)}]
 The rank of $\widehat{M}^k$ is equal to the  appreciable  rank of $\widehat{M}^k$;

  \item[{\bf  (5)}]
 The index of $\widehat{M}$ is equal to the  appreciable  index of $\widehat{M}$;

  \item[{\bf  (6)}]
   $\left(I_n-MM^{D}\right)K_0\left(I_n-MM^{D}\right)=O$,
   in which
  $K_0=\overset{k}{\underset{i=1}\sum}M^{k-i}M_0M^{i-1}$.
\end{enumerate}
\end{remark}

\begin{remark}
\label{Remark-GDDI-3}
   From Theorem \ref{20220312-Th-1},
Theorem \ref{Th3.2}
 and Definition  \ref{GDDI-Def},
we see that any  square dual matrix has its own unique WDDI.
And since $\widehat{M}^{D}=\widehat{M}_{\rm W}^{D}$,
when the DDI of $\widehat{M}$ exists,
then the WDDI is a generalization of DDI.
\end{remark}

The following example illustrates that $\widehat{M}^{D}=\widehat{M}_{\rm W}^{D}$,
when the DDI of $\widehat{M}$ exists.

\begin{example}
\label{Ex-3-2}
Let
$\widehat M=\begin{bmatrix}
        4&8&12&10  \\
        2&8&10&8   \\
         0&-2&-2&0   \\
        -2&-4&-6&-6     \end{bmatrix}
+\varepsilon \begin{bmatrix}
        -4&3&-3&2  \\
        5&4&0&2   \\
        -7&7&1&0   \\
        2&-3&2&1     \end{bmatrix}$.
It is easy to know that ${\mbox{\rm AInd}} \left(\widehat{M}\right)=k=2 $,
\begin{align*}
M^D&=\begin{bmatrix}
        \frac32&-\frac12&1&1  \\
        1&\frac12&\frac32&\frac32   \\
        -\frac12&0&-\frac12&-\frac12   \\
        -\frac12&0&-\frac12&-\frac12 \end{bmatrix},
M^DA_0M^D=\begin{bmatrix}
        \frac{29}2&10&-\frac92&-\frac92  \\
        -\frac{11}2&10&\frac92&\frac92  \\
        4&-4&0&0   \\
        4&-4&0&0        \end{bmatrix},\\
   K_0&=\overset{2}{\underset{i=1}\sum}M^{2-i}M_0M^{i-1}=\begin{bmatrix}
        -54&88&-4&6  \\
        2&148&108&98  \\
        -10&-24&-18&-18   \\
        18&-62&-28&-28       \end{bmatrix},
\end{align*}
         then
\begin{align*}
\left(I_4-MM^D\right)K_0\left(I_4-MM^D\right)=O.
\end{align*}
Hence,   by applying  Lemma \ref{le-1.1}
we can know  that the DDI of $\widehat{M}$  exists,
and
\begin{align*}
S&={\left(M^D\right)}^2
\left(\overset{1}{\underset{i=0}\sum}{\left(M^D\right)}^iM_0M^i\right)
\left(I_4-MM^D\right)
\\
&\qquad\qquad
+\left(I_4-MM^D\right)
\left(\overset{1}{\underset{i=0}\sum}M^iM_0{\left(M^D\right)}^i\right)
{\left(M^D\right)}^2-M^DM_0M^D
\\
&=\begin{bmatrix}
        48&-\frac{55}4&\frac{163}4&\frac{213}4  \\
        \frac{97}4&-\frac{57}2&-2&\frac{17}4   \\
        -17&19&\frac{1}4&-\frac{7}2   \\
        -\frac{27}2&\frac{15}4&-\frac{23}2&-\frac{61}4       \end{bmatrix}.
\end{align*}
Then
\begin{align*}
\widehat{M}^{D}=\begin{bmatrix}
         \frac32&-\frac12&1&1  \\
        1&\frac12&\frac32&\frac32   \\
        -\frac12&0&-\frac12&-\frac12   \\
        -\frac12&0&-\frac12&-\frac12     \end{bmatrix}
+\varepsilon \begin{bmatrix}
        48&-\frac{55}4&\frac{163}4&\frac{213}4  \\
        \frac{97}4&-\frac{57}2&-2&\frac{17}4   \\
        -17&19&\frac{1}4&-\frac{7}2   \\
        -\frac{27}2&\frac{15}4&-\frac{23}2&-\frac{61}4       \end{bmatrix}.
\end{align*}

 By  (\ref{20231016-Eq-8}) we have
\begin{align*}
\widehat{M}_{\rm W}^{D}=\begin{bmatrix}
         \frac32&-\frac12&1&1  \\
        1&\frac12&\frac32&\frac32   \\
        -\frac12&0&-\frac12&-\frac12   \\
        -\frac12&0&-\frac12&-\frac12     \end{bmatrix}
+\varepsilon \begin{bmatrix}
        48&-\frac{55}4&\frac{163}4&\frac{213}4  \\
        \frac{97}4&-\frac{57}2&-2&\frac{17}4   \\
        -17&19&\frac{1}4&-\frac{7}2   \\
        -\frac{27}2&\frac{15}4&-\frac{23}2&-\frac{61}4       \end{bmatrix}=\widehat{M}^{D}.
\end{align*}
\end{example}

\

In Examples \ref{Ex-3-1} and \ref{Ex-3-2},
the method for calculating WDDIs is based on (\ref{GDDI -def-Solution-1})  and (\ref{20231016-Eq-8})  in  Theorem \ref{Th3.2}.
It is worth noting that, in the process of computing,
the standard and dual parts of WDDI are provided separately.
By nature, this is transforming a dual matrix problem into a real matrix problem.

Wang et al. \cite{Wang2023LaaSub}
introduce the decomposition (\ref{dual-C-N}) of a square dual matrix.
In following Theorem \ref{GDDI-Def-2}, by applying the decomposition we get a new characterization of the WDDI.
Here, we do not list the dual and real parts of the square matrix $\widehat M$.

\begin{theorem}
\label{GDDI-Def-2}
Let
$\widehat  M    \in \mathbb{D}^{n\times n}$,
${\mbox{\Dind}}\left(\widehat{M}\right)=t$,
${\Drk}\left( \widehat{M}^t \right)=r$
and
  the  decomposition of $\widehat  M$ be given as in (\ref{dual-C-N}).
Then
\begin{align}
\label{GDDI-Def-Char-2}
\widehat{M}_{\rm W}^{D}
=
\widehat{P}
\begin{bmatrix}
\widehat{C}^{-1} & 0 \\
0 & 0
\end{bmatrix}
 \widehat{P}^{-1},
\end{align}
where
$\widehat{P} \in \mathbb{D}^{n \times n}$
 and
$\widehat{C} \in \mathbb{D}^{r \times r}$ are invertible.
\end{theorem}

\begin{proof}
Let
$\widehat  M = M + \varepsilon M_0 \in\mathbb{D}^{n\times n}$,
${\mbox{\Dind}}\left(\widehat{M}\right)=t$,
${\Drk}\left( \widehat{M}^t \right)=r$
and
  the  decomposition of $\widehat  M$ be given as in (\ref{dual-C-N}).
Since $\widehat N^{t}=O$,
then
$$
\widehat{M}^t
=
\widehat{P}
\begin{bmatrix}
\widehat{C}^t & O \\
O & \widehat{N}^t
\end{bmatrix}
 \widehat{P}^{-1}
=
\widehat{P}
\begin{bmatrix}
\widehat{C}^t & O \\
O & O
\end{bmatrix}
 \widehat{P}^{-1}.
$$
Furthermore, suppose that
$$\widehat{X}_0
=\widehat{P}\begin{bmatrix}
\widehat{C}^{-1} & O \\
O & O
\end{bmatrix} \widehat{P}^{-1}.$$
Then we get
\begin{align*}
\widehat{M}\widehat{X}_0\widehat{M}^{t}
&
=
\widehat{P}\begin{bmatrix}
\widehat{C} & O \\
O & \widehat{N}
\end{bmatrix}\begin{bmatrix}
\widehat{C}^{-1} & O \\
O & O
\end{bmatrix}\begin{bmatrix}
\widehat{C}^{t} & O \\
O & O
\end{bmatrix}  \widehat{P}^{-1}
=\widehat{P}  \begin{bmatrix}
\widehat{C}^{t} & O \\
O &O
\end{bmatrix} \widehat{P}^{-1}
=\widehat{M}^{t},
\\
\widehat{X}_0 \widehat{M} \widehat{X}_0
&
=
\widehat{P}\begin{bmatrix}
\widehat{C}^{-1} & O \\
O & O
\end{bmatrix}\begin{bmatrix}
\widehat{C} & O \\
O & \widehat{N}
\end{bmatrix}\begin{bmatrix}
\widehat{C}^{-1} & O \\
O & O
\end{bmatrix} \widehat{P}^{-1}=\widehat{P}
\begin{bmatrix}
\widehat{C}^{-1} & O \\
O & O
\end{bmatrix} \widehat{P}^{-1}=\widehat{X}_0 .
\\
 \widehat{M} \widehat{X}_0
 &
 =
 \widehat{P}\begin{bmatrix}
\widehat{C} & O \\
O & \widehat{N}
\end{bmatrix}\begin{bmatrix}
\widehat{C}^{-1} & O \\
O & O
\end{bmatrix} \widehat{P}^{-1}=\widehat{P}\begin{bmatrix}
I_{r} & O \\
O & O
\end{bmatrix} \widehat{P}^{-1},
\\
 \widehat{X}_0 \widehat{M}
 &
 =
 \widehat{P}\begin{bmatrix}
\widehat{C}^{-1} & O \\
O & O
\end{bmatrix}\begin{bmatrix}
\widehat{C} & O \\
O & \widehat{N}
\end{bmatrix} \widehat{P}^{-1}=\widehat{P}\begin{bmatrix}
I_{r} & O \\
O & O
\end{bmatrix} \widehat{P}^{-1}.
\end{align*}
Hence $ \widehat{X}_0$ is the solution to (\ref{GDGI-def}).
Furthermore, by Theorem \ref{Th3.2}
we have
$\widehat{X}_0=\widehat{A}_{\rm W}^{D}$,
that is, (\ref{GDDI-Def-Char-2}).
\end{proof}

 \section{Weak dual group inverse}
Zhong and Zhang \cite{zhong--dual Drazin inverse} consider  the DGI.
Next we consider a generalized DGI.

\begin{definition}
\label{GDGI-def-EQ}
Let $\widehat{M} \in \mathbb{D}^{n \times n}$ with  ${\mbox{\Aind}}\left({\widehat M}\right)=1$.
 Then  the solution to
\begin{align}
\label{GDGI-def-Eq}
\widehat M\widehat X\widehat M^{2}=\widehat M^{2},\
 \widehat X\widehat M\widehat X=\widehat X, \
 \widehat M\widehat X=\widehat X\widehat M
\end{align}
 is unique.
We call the unique solution  the weak DGI (WDGI) of $\widehat{M}$,
and denote it by $\widehat X=\widehat{M}_{\rm W}^{\#}$.
\end{definition}

\begin{theorem}
\label{Th-GDDI-EQ}
Let $\widehat{M}=M+\varepsilon M_{0} \in \mathbb{D}^{n \times n}$
with
 ${ {\Aind}}\left({\widehat M}\right)=1$,
 then
\begin{align}
\label{GDDI-eq}
\widehat{M}_{\rm W}^{\#}=M^{\#}+\varepsilon R,
\end{align}
where
$ R={\left(M^{\#}\right)}^2M_0\left(I_n-MM^{\#}\right)
+\left(I_n-MM^{\#}\right)M_0{\left(M^{\#}\right)}^2-M^{\#}M_0M^{\#}$.
\end{theorem}

\begin{proof}
Let $\widehat{M}=M+\varepsilon M_{0} \in \mathbb{D}^{n \times n}$,
with   ${ {\Aind}}\left({\widehat M}\right)=1$ and $\rk(M)=r$.
Since ${ {\Aind}}\left({\widehat M}\right)=1$,
then ${ {\ind}}\left({ M}\right)=1$.
By applying Theorem \ref{C-N-Decomposion-Th},
we get
\begin{align}
\label{20231016-Eq-9}
 M=P\begin{bmatrix}
        C&O  \\
        O&O         \end{bmatrix}P^{-1},
        \end{align}
where $P$ and $C$ are nonsingular.
Write
\begin{align}
\label{20231016-Eq-10}
M_0=P\begin{bmatrix}
          M_1&M_2  \\
          M_3&M_4  \end{bmatrix}P^{-1},
        \end{align}
where $  M_1\in \mathbb{C}^{r \times r}$.
Further, denote
\begin{align}
\label{20231016-Eq-11}
\widehat H=P\begin{array}{c}\begin{bmatrix}
C^{-1}&O  \\
     O&O   \end{bmatrix}\end{array}P^{-1}+\varepsilon
P\begin{array}{c}\begin{bmatrix}
     -C^{-1}M_1C^{-1}    & C^{-2}M_2 \\
     M_3C^{-2} &    O               \end{bmatrix}\end{array}P^{-1}.
        \end{align}

By applying (\ref{20231016-Eq-9}), (\ref{20231016-Eq-10}) and (\ref{20231016-Eq-11}),
we have
\begin{align*}
\widehat M\widehat H
&=P\begin{bmatrix}
     I_r&O  \\
     O&O\end{bmatrix}P^{-1}
+\varepsilon P\begin{array}{c}\begin{bmatrix}
     O&C^{-1}M_2    \\
      M_3C^{-1}&O                     \end{bmatrix}\end{array}P^{-1}=
\widehat H\widehat M ,
\\
\widehat M\widehat H \widehat M^{2}
&=P\begin{bmatrix}
        C^{2}&O  \\
        O&O   \end{bmatrix}P^{-1}
+\varepsilon P\begin{bmatrix}
CM_1+M_1C&  CM_2    \\
M_3C&O                           \end{bmatrix}P^{-1}
=\widehat M^{2},
\\
\widehat H\widehat M\widehat H
&
=P\begin{array}{c}\begin{bmatrix}
C^{-1}&O  \\
     O&O   \end{bmatrix}\end{array}P^{-1}+\varepsilon
P\begin{array}{c}\begin{bmatrix}
     -C^{-1}M_1C^{-1}    & C^{-2}M_2 \\
     M_3C^{-2} &    O               \end{bmatrix}\end{array}P^{-1}
=\widehat H.
\end{align*}
Therefore,
from Definition \ref{GDGI-def-EQ},
we get that
$\widehat M_{\rm W}^{\#}$ exists
 and
$\widehat M_{\rm W}^{\#}=\widehat H$.
Furthermore, from (\ref{20231016-Eq-11}) and
\begin{align}
M^{\#}M_0M^{\#}
&
=
 P\begin{array}{c}\begin{bmatrix}
     -C^{-1}M_1C^{-1}  &   O \\
     O &    O               \end{bmatrix}\end{array}P^{-1},
\\
{\left(M^{\#}\right)}^2
M_0
\left(I_n-MM^{\#}\right)
&
=
P\begin{array}{c}\begin{bmatrix}
     O               &   C^{-2}M_2 \\
     O &    O               \end{bmatrix}\end{array}P^{-1},
\\
\left(I_n-MM^{\#}\right)
M_0
{\left(M^{\#}\right)}^2
&
=
P\begin{array}{c}\begin{bmatrix}
     O               &   O \\
    M_3C^{-2} &    O               \end{bmatrix}\end{array}P^{-1},
\end{align}
it follows  that we get (\ref{GDDI-eq}).
\end{proof}

From  (\ref{GDGI-def}), Definition \ref{GDDI-Def},  (\ref{GDGI-def-Eq})  and Definition \ref{GDGI-def-EQ},
it easy to check that the WDGI is a special case of the WDDI.
The difference is that in Definition \ref{GDGI-def-EQ}, we apply the appreciable index rather than the index of the dual matrix.
Next, an example is presented to show that the WDGI of $\widehat M$ always exists,
in which $\widehat M\in \mathbb{D}^{n \times n}$ and ${\mbox{\rm AInd}} \left(\widehat{M}\right)=1$.

\begin{example}
Let $\widehat M=M+\varepsilon M_0=\begin{bmatrix}
        1&0  \\
        0&0        \end{bmatrix}
+\varepsilon \begin{bmatrix}
        0&0  \\
        1&1        \end{bmatrix}$,
 then ${\mbox{\rm AInd}} \left(\widehat{M}\right)=1 $.
It is obvious that
$$M^{\#}=\begin{bmatrix}
        1&0  \\
        0&0        \end{bmatrix},
$$
and
$$
\left(I_2-MM^{\#}\right)M_0\left(I_2-MM^{\#}\right)=
\begin{bmatrix}
        0&0  \\
        0&1        \end{bmatrix}
\begin{bmatrix}
        0&0  \\
        1&1        \end{bmatrix}
\begin{bmatrix}
        0&0  \\
        0&1        \end{bmatrix}
=\begin{bmatrix}
        0&0  \\
        0&1        \end{bmatrix}
\neq O.
$$
Hence, from Lemma \ref{le-DGI-CGGI}
we see that the DGI of $\widehat{M}$ is non-existent,
but  the WDGI of $\widehat{M}$ is existent.
From Theorem \ref{Th-GDDI-EQ}, we have
$$\begin{aligned}
R&={\left(M^{\#}\right)}^2M_0\left(I_2-MM^{\#}\right)
+\left(I_2-MM^{\#}\right)M_0{\left(M^{\#}\right)}^2-M^{\#}M_0M^{\#}
 =\begin{bmatrix}
        0&0  \\
        1&0        \end{bmatrix}.
\end{aligned}$$
It follows that
$$\widehat{M}_{\rm W}^{\#}=M^{\#}+\varepsilon R
=\begin{bmatrix}
        1&0  \\
        0&0        \end{bmatrix}
+\varepsilon \begin{bmatrix}
        0&0  \\
        1&0        \end{bmatrix}.$$
\end{example}

\begin{remark}
\label{Remark-4-2}
From (\ref{DCI-Def}),
Lemma \ref{le-DGI-CGGI},
Definition  \ref{GDGI-def-EQ}
 and Theorem \ref{Th-GDDI-EQ},
it is  easy to verify that   $\widehat{M}^{\#}=\widehat{M}_{\rm W}^{\#}$,
when the DGI of $\widehat{M}$ exists.
\end{remark}

\begin{example}
\label{Ex-4-3}
Let $\widehat M=M+\varepsilon M_0=\begin{bmatrix}
        1&0  \\
        0&0        \end{bmatrix}
+\varepsilon \begin{bmatrix}
        0&1  \\
        1&0        \end{bmatrix}$,
then we have $\ind (M)=1$,
$$ M^{\#}=\begin{bmatrix}
        1&0  \\
        0&0        \end{bmatrix}
$$
and
$$
\left(I_2-MM^{\#}\right)M_0\left(I_2-MM^{\#}\right)=
\begin{bmatrix}
        0&0  \\
        0&1        \end{bmatrix}
\begin{bmatrix}
        0&1  \\
        1&0        \end{bmatrix}
\begin{bmatrix}
        0&0  \\
        0&1        \end{bmatrix}
=O.
$$
Hence, from Lemma \ref{le-DGI-CGGI}
we can know  that DGI of $\widehat{M}$ exists and
\begin{align*}
\widehat{M}^{\#}
&=M^{\#}+\varepsilon
\left({\left(M^{\#}\right)}^2M_0\left(I_2-MM^{\#}\right)
+\left(I_2-MM^{\#}\right)M_0{\left(M^{\#}\right)}^2-M^{\#}M_0M^{\#}\right)
\\
&=\begin{bmatrix}
        1&0  \\
        0&0        \end{bmatrix}
+\varepsilon \begin{bmatrix}
        0&1  \\
        1&0        \end{bmatrix}.
\end{align*}
Applying Theorem \ref{Th-GDDI-EQ}  we have
$$
R={\left(M^{\#}\right)}^2M_0\left(I_2-MM^{\#}\right)
+\left(I_2-MM^{\#}\right)M_0{\left(M^{\#}\right)}^2-M^{\#}M_0M^{\#}
 =\begin{bmatrix}
        0&1  \\
        1&0        \end{bmatrix}.$$
Therefore,
$$\widehat{M}_{\rm W}^{\#}=M^{\#}+\varepsilon R
=\begin{bmatrix}
        1&0  \\
        0&0        \end{bmatrix}
+\varepsilon \begin{bmatrix}
        0&1  \\
        1&0        \end{bmatrix}
=\widehat{M}^{\#}.
$$
\end{example}

 \

From Lemma \ref{dual-C-N},
we know that when $\Aind\left(\widehat{M}\right)=1$,
then the   standard part of $\widehat{N}$ is $O$.
Denote  $\widehat{N}=\varepsilon N$, in which $N\in \mathbb{R}^{(n-r) \times (n-r)}$.
Therefore,
when
$\widehat  M \in\mathbb{D}^{n\times n}$
with
$\Aind\left(\widehat{M}\right)=1$ and
${\Ark}\left( \widehat{M}  \right)=r$,
  there exists an invertible dual matrix
$\widehat  P  \in\mathbb{D}^{n\times n}$
such that $\widehat  M$  has the following form:
\begin{align}
\label{C-N-D-Ind=1}
\widehat M =
\widehat P
\begin{bmatrix}
 \widehat C &       O     \\
        O   & \varepsilon N
\end{bmatrix}\widehat P^{-1},
\end{align}
where  $\widehat{C} \in \mathbb{D}^{r \times r}$ is  invertible
and
$N\in \mathbb{R}^{(n-r) \times (n-r)}$.
It follows that from Definition \ref{GDGI-def-Eq} that
\begin{align}
\label{C-N-D-Ind=1-WDGI}
\widehat M_{\rm W}^{\#} =
\widehat P
\begin{bmatrix}
 \widehat C ^{-1} &       O     \\
        O        &    O
\end{bmatrix}\widehat P^{-1}.
\end{align}

Next, we consider the dual matrix equation
\begin{align}
\label{DL-eQ-G-1}
\widehat{M}\widehat{x}=\widehat{b},
\end{align}
in which  $\widehat M \in\mathbb{D}^{n\times n}$,
${\Ark}\left( \widehat{M}  \right)=r$,
${ {\Aind}}\left(\widehat {M}\right)=1$ and
$\widehat b \in\mathbb{D}^{n\times 1}$.

\begin{theorem}
\label{App-G-1}
Let $\widehat M\in\mathbb{D}^{n\times n}$, 
${ {\Aind}}\left(\widehat {M}\right)=1$, $\Ark\left(\widehat {M}\right)=r$ and
$\widehat b\in\mathbb{D}^{n\times 1}$. 
Then (\ref{DL-eQ-G-1}) is consistent if and only if
 \begin{numcases}{}
\label{App-G-1-E-1}
\varepsilon\left(I_n-\widehat{M}_{\rm W}^{\#}\widehat{M}\right)\widehat{b} =O,
\\
\label{App-G-1-E-2}
 \left(I_n-\widehat{M}_{\rm W}^{\#}\widehat{M}\right)\widehat{b}
\in \mathcal{R}\left(\widehat{M}-\left(\widehat{M}_{\rm W}^{\#}\right)^{\#}\right).
\end{numcases}

Furthermore, let the decomposition of $\widehat{M}$ be as in (\ref{C-N-D-Ind=1})
and
  $\widehat P^{-1}\widehat{b}$ be partitioned as
\begin{align}
\label{App-G-1-E-1-Proof-4}
\widehat P^{-1}\widehat{b}
=
\begin{bmatrix}
 \widehat{b}_1    \\
 \widehat{b}_2
\end{bmatrix}.
\end{align}
Then
the general solution to (\ref{DL-eQ-G-1}) is
 \begin{align}
\label{App-G-1-E-5}
 \widehat{x}
=
 \widehat P
 \begin{bmatrix}
 \widehat{C}^{-1} \widehat{b}_1   \\
 N^{\dag}\widehat{b}_2+ \left(I_{n-r}-N^{\dag}N\right)y_2+ \varepsilon{z}_2
\end{bmatrix},
\end{align}
where $y_2 \in\mathbb{R}^{(n-r)\times 1}$ and  $z_2 \in\mathbb{R}^{(n-r)\times 1}$ are arbitrary.
\end{theorem}
\begin{proof}
Let $\widehat M \in\mathbb{D}^{n\times n}$,
${ {\Aind}}\left(\widehat {M}\right)=1$,  $\Ark\left(\widehat {M}\right)=r$
and the decomposition of $\widehat {M}$ be as in (\ref{C-N-D-Ind=1}).
It is easy to check that
 \begin{numcases}{}
\label{App-G-1-E-1-Proof-3}
\begin{aligned} 
  I_n-\widehat{M}_{\rm W}^{\#}\widehat{M}
  &
  =
  \widehat P
\begin{bmatrix}
O   &    O     \\
 O  & I_{n-r}
\end{bmatrix}\widehat P^{-1}, 
\\
 \left(\widehat{M}_{\rm W}^{\#}\right)^{\#}
  &
 =\widehat P
\begin{bmatrix}
\widehat{C} &       O     \\
 O   & O
\end{bmatrix}\widehat P^{-1}
\\
 \widehat{M}-\left(\widehat{M}_{\rm W}^{\#}\right)^{\#}
  &
 =\widehat P
\begin{bmatrix}
O &       O     \\
 O   & \varepsilon N
\end{bmatrix}\widehat P^{-1}.
\end{aligned} \end{numcases}
Furthermore,
let $\widehat P^{-1}\widehat{x}$ be partitioned as
\begin{align}
\label{App-G-1-E-1-Proof-0}
\widehat P^{-1}\widehat{x}= \begin{bmatrix}
 \widehat{x}_1    \\
 \widehat{x}_2
\end{bmatrix}.
\end{align}
It follows from   (\ref{App-G-1-E-1-Proof-4}) and (\ref{App-G-1-E-1-Proof-0})  that
\begin{align}
\nonumber
\widehat M \widehat{x}- \widehat{b}
&=
\widehat P
\left(\begin{bmatrix}
 \widehat C &       O     \\
        O   & \varepsilon N
\end{bmatrix}\widehat P^{-1}\widehat{x}-\widehat P^{-1}\widehat{b}\right)
\\
\label{App-G-1-E-1-Proof-1}
& =
\widehat P
\left(\begin{bmatrix}
 \widehat C &       O     \\
        O   & \varepsilon N
\end{bmatrix} \begin{bmatrix}
 \widehat{x}_1    \\
 \widehat{x}_2
\end{bmatrix}-\begin{bmatrix}
 \widehat{b}_1    \\
 \widehat{b}_2
\end{bmatrix}\right)
=
\widehat P
\left(  \begin{bmatrix}
 \widehat C \widehat{x}_1- \widehat{b}_1   \\
 \varepsilon N\widehat{x}_2- \widehat{b}_2
\end{bmatrix}\right).
\end{align}
Therefore,
$\widehat M \widehat{x}=\widehat{b}$ is consistent if and only if
there exist $\widehat{x}_1$ and $\widehat{x}_2$ satisfying
 $\widehat C \widehat{x}_1= \widehat{b}_1 $ and
 $\varepsilon N\widehat{x}_2= \widehat{b}_2$.

Since  $\widehat C $ is invertible,
 then it is obvious that $\widehat C \widehat{x}_1= \widehat{b}_1 $ is consistent.
%
%

Since
 $\varepsilon N\widehat{x}_2= \widehat{b}_2$,
 then the standard part of  $\widehat{b}_2$ is equal to $O$,
 that is,
 $\varepsilon\widehat{b}_2=O$.
Furthermore,
by applying  (\ref{App-G-1-E-1-Proof-4}),   (\ref{App-G-1-E-1-Proof-3}) and (\ref{App-G-1-E-1-Proof-0}),
we have
$\varepsilon\left(I_n-\widehat{M}_{\rm W}^{\#}\widehat{M}\right)\widehat{b}=O$,
that is,
 (\ref{App-G-1-E-1}).
Besides, when $\varepsilon\widehat{b}_2=O$,
there exists   $\widehat{x}_2$ satisfying
 $\varepsilon N\widehat{x}_2= \widehat{b}_2$
 if and only if
 $ \left(I_n-\widehat{M}_{\rm W}^{\#}\widehat{M}\right)\widehat{b}
\in
\mathcal{R}\left(\widehat{M}-\left(\widehat{M}_{\rm W}^{\#}\right)^{\#}\right)$,
that is (\ref{App-G-1-E-2}).

Let (\ref{DL-eQ-G-1}) be consistent,
then from (\ref{App-G-1-E-1-Proof-4}) and  (\ref{App-G-1-E-1-Proof-1})
it is obvious that
 $\widehat{x}_1=\widehat{C}^{-1} \widehat{b}_1 $
 and
$ \widehat{x}_2=N^{\dag}\widehat{b}_2+ \left(I-N^{\dag}N\right)y_2+ \varepsilon z_2$,
in which $y_2 \in\mathbb{R}^{(n-r)\times 1}$ and  $z_2 \in\mathbb{R}^{(n-r)\times 1}$ are arbitrary.
Therefore,
it follows from (\ref{App-G-1-E-1-Proof-0})
that we get (\ref{App-G-1-E-5}).
\end{proof}

We can derive the following Corollary \ref{App-G-1-C} from the Theorem \ref{App-G-1}.

\begin{corollary}[\cite{zhong--dual group inverse}]
\label{App-G-1-C}
Let $\widehat M\in\mathbb{D}^{n\times n}$, 
${ {\Dind}}\left(\widehat {M}\right)=1$ and
$\widehat b\in\mathbb{D}^{n\times 1}$.
Then (\ref{DL-eQ-G-1}) is consistent if and only if
$ \left(I_n-\widehat{M}^{\#}\widehat{M}\right)\widehat{b} =O$.

Furthermore,
the general solution to (\ref{DL-eQ-G-1}) is
$
 \widehat{x}
=\widehat{M}^{\#}\widehat{b}+
 \left(I_n-\widehat{M}^{\#}\widehat{M}\right)\widehat{y}$,
where $\widehat{y}  \in\mathbb{D}^{n\times 1}$ is  arbitrary.
\end{corollary}

\begin{proof}
Let $\widehat M \in\mathbb{D}^{n\times n}$,
${ {\Dind}}\left(\widehat {M}\right)=1$,  $\Ark\left(\widehat {M}\right)=r$
and the decomposition of $\widehat {M}$ be as in (\ref{C-N-D-Ind=1}).
Since   ${ {\Dind}}\left(\widehat {M}\right)=1$,
then
$N=O$.
It follows from (\ref{App-G-1-E-1-Proof-1}) that
$\widehat M \widehat{x}=\widehat{b}$ is consistent if and only if
there exists
 $  \widehat{b}_2=O$.
 Therefore,
by applying (\ref{App-G-1-E-1-Proof-3}) we get that
$  \widehat{b}_2=O$
if and only if
$\left(I_n-\widehat{M}^{\#}\widehat{M}\right)\widehat{b} =O$.

Since $N=O$,
then by applying (\ref{App-G-1-E-5}) we get
$ \widehat{x}
=
 \widehat P
 \begin{bmatrix}
 \widehat{C}^{-1} \widehat{b}_1   \\
y_2+ \varepsilon{z}_2
\end{bmatrix}$,
where $y_2 \in\mathbb{R}^{(n-r)\times 1}$ and  $z_2 \in\mathbb{R}^{(n-r)\times 1}$ are arbitrary.
It follows from  (\ref{App-G-1-E-1-Proof-3})  that
$ \widehat{x}
=\widehat{M}^{\#}\widehat{b}+
 \left(I_n-\widehat{M}^{\#}\widehat{M}\right)\widehat{y}$,
where $\widehat{y}  \in\mathbb{D}^{n\times 1}$ is  arbitrary.
\end{proof}

 \

In \cite{Campbell2009book},
Campbell and  Meyer get that $x=M^\# b$
 is the unique solution of the restricted   matrix equation $Mx=b$ with
$ {x}\in \mathcal{R}( {M})$ and $\ind(M)=1$.
In \cite{zhong--dual group inverse},
by applying the DGI, Zhong and Zhang consider the restricted dual matrix equation  (\ref{DL-eQ-G-1})
with
$\widehat{x}\in \mathcal{R}(\widehat{M})$
and
 $\Dind\left(\widehat{M}\right)=1$.
In the following Theorem \ref{App-G-2},
by applying the WDGI we consider
the general restricted dual matrix equation (\ref{DL-eQ-G-1})
with
$\widehat{x}\in \mathcal{R}(\widehat{M})$
and
$\Aind\left(\widehat{M}\right)=1$.

\begin{theorem} \label{App-G-2}
Let $\widehat M \in\mathbb{D}^{n\times n}$,
${ {\Aind}}\left(\widehat {M}\right)=1$ and
$\widehat b \in\mathbb{D}^{n\times 1}$.
 Then the  solution to  (\ref{DL-eQ-G-1}) with
 $\widehat{x}\in \mathcal{R}(\widehat{A})$
 exists
 if and only if
 \begin{align}
\label{App-G-1-E-3}
\left(I_n-\widehat{M}_{\rm W}^{\#}\widehat{M}\right)\widehat{b}&=O.
\end{align}
Furthermore, when the restricted  solution to  (\ref{DL-eQ-G-1}) exists,
 \begin{align}
\label{App-G-1-E-4}
 \widehat{x}=\widehat{M}_{\rm W}^{\#}\widehat{b} +
 \left(\widehat{M}-\left(\widehat{M}_{\rm W}^{\#}\right)^{\#}\right)\widehat{y},
\end{align}
where  $\widehat{y}\in\mathbb{D}^{n\times 1}$ is arbitrary.
\end{theorem}

\begin{proof}
Let $\widehat M=M+\varepsilon M_0\in\mathbb{D}^{n\times n}$,
${ {\Aind}}\left(\widehat {M}\right)=1$, $\Ark\left(\widehat {M}\right)=r$
and the decomposition of $\widehat {M}$ be as in (\ref{C-N-D-Ind=1}).
Since
$\widehat{x}\in \mathcal{R}\left(\widehat{M}\right)$,
then we write $\widehat{x}= \widehat{M}\widehat{y}$.
Therefore,
the restricted dual matrix equation   (\ref{DL-eQ-G-1}) is consistent  with $\widehat{x}\in \mathcal{R}\left(\widehat{M}\right)$
if and only if
the dual matrix equation $ \widehat{M}^2\widehat{y}=\widehat{b} $ is consistent.
Let $\widehat P^{-1}\widehat{b}$
be partitioned as in  (\ref{App-G-1-E-1-Proof-4}),
and
 $\widehat P^{-1}\widehat{y}$ be partitioned as
\begin{align}
\label{App-G-1-E-1-Proof-5}
\widehat P^{-1}\widehat{y}= \begin{bmatrix}
 \widehat{y}_1    \\
 \widehat{y}_2
\end{bmatrix}.
\end{align}
It follows that
\begin{align}
\label{App-G-1-E-1-Proof-2}
 \widehat{M}^2\widehat{y}-\widehat{b}
 =
\widehat P
\left(\begin{bmatrix}
 \widehat{C}^2 &       O     \\
        O   & O
\end{bmatrix} \begin{bmatrix}
 \widehat{y}_1    \\
 \widehat{y}_2
\end{bmatrix}-\begin{bmatrix}
 \widehat{b}_1    \\
 \widehat{b}_2
\end{bmatrix}\right)
=
\widehat P
\left(  \begin{bmatrix}
 \widehat{C}^2\widehat{y}_1- \widehat{b}_1   \\
 - \widehat{b}_2
\end{bmatrix}\right).
 \end{align}

Since $\widehat{C}$ is invertible,
then $ \widehat{M}^2\widehat{y}=\widehat{b} $ is consistent
if and only if
$\widehat{b}_2=O$.
Therefore,
$\left(I_n-\widehat{M}_{\rm W}^{\#}\widehat{M}\right)\widehat{b}=O$,
that is,
(\ref{App-G-1-E-3}).

Furthermore, let the restricted solution to  (\ref{DL-eQ-G-1}) exist,
then
$$ \widehat{y}=
\widehat P\begin{bmatrix}
\widehat{C}^{-2} \widehat{b}_1   \\
 \widehat{y}_2
\end{bmatrix}.$$
It follows that
\begin{align}
\nonumber
 \widehat{x}
 &
 =
 \widehat{M}\widehat{y}=
 \widehat P
\begin{bmatrix}
 \widehat{C}  &       O     \\
        O   & \varepsilon N
\end{bmatrix}
\begin{bmatrix}
\widehat{C}^{-2} \widehat{b}_1   \\
 \widehat{y}_2
\end{bmatrix}
=
 \widehat P
\begin{bmatrix}
 \widehat{C}^{-1}  &       O     \\
        O   &  O
\end{bmatrix}
\begin{bmatrix}
  \widehat{b}_1   \\
 \widehat{b}_2
\end{bmatrix}
+
 \widehat P \begin{bmatrix}
O &O\\
O &\varepsilon N
\end{bmatrix}
\begin{bmatrix}
  \widehat{y}_1   \\
 \widehat{y}_2
\end{bmatrix}
\\
\nonumber
 &
 =
 \widehat{M}_{\rm W}^{\#}\widehat{b} +
 \left(\widehat{M}-\left(\widehat{M}_{\rm W}^{\#}\right)^{\#}\right)\widehat{y},
\end{align}
that is, (\ref{App-G-1-E-4}).
\end{proof}

\begin{example}
\label{Ex-4-2}
Let $\widehat M=M+\varepsilon M_0=\begin{bmatrix}
        1&0  \\
        0&0        \end{bmatrix}
+\varepsilon \begin{bmatrix}
        0&0  \\
        1&1        \end{bmatrix}$
        and
        $\widehat b=\begin{bmatrix}
        1   \\
        0        \end{bmatrix}
+\varepsilon \begin{bmatrix}
        0  \\
        1        \end{bmatrix}$.
It is not difficult to verify that
 \begin{align}
 \nonumber
 \widehat{M}_{\rm W}^{\#}=M^{\#}+\varepsilon R
=\begin{bmatrix}
        1&0  \\
        0&0        \end{bmatrix} +\varepsilon \begin{bmatrix}
        0&0  \\
        1&0        \end{bmatrix}
\end{align}
and
 \begin{align}
 \nonumber
\left(I_2-\widehat{M}_{\rm W}^{\#}\widehat{M}\right)\widehat{b}
=\left(\begin{bmatrix}
        0&0  \\
        0&1        \end{bmatrix} +\varepsilon \begin{bmatrix}
        0&0  \\
       - 1&0        \end{bmatrix} \right)\left(\begin{bmatrix}
        1   \\
        0        \end{bmatrix}
+\varepsilon \begin{bmatrix}
        0  \\
        1        \end{bmatrix}\right)
=O.
\end{align}
Then the  solution to  (\ref{DL-eQ-G-1}) with
 $\widehat{x}\in \mathcal{R}(\widehat{M})$
 exists
 and
  \begin{align}
 \nonumber
 \widehat{x}
 &
 =
 \widehat{M}_{\rm W}^{\#}\widehat{b} +
 \left(\widehat{M}-\left(\widehat{M}_{\rm W}^{\#}\right)^{\#}\right)\widehat{y}
 =
 \begin{bmatrix}
        1   \\
        0        \end{bmatrix}
+\varepsilon \begin{bmatrix}
        0  \\
        1        \end{bmatrix}
        +
\varepsilon \begin{bmatrix}
        0&0  \\
        1&1        \end{bmatrix}\widehat{y},
\end{align}
where  $\widehat{y}\in\mathbb{D}^{n\times 1}$ is arbitrary.

When  $\widehat{y}=O$,
then
$$\widehat{x}_1
 =
 \begin{bmatrix}
        1   \\
        0        \end{bmatrix}
+\varepsilon \begin{bmatrix}
        0  \\
        1        \end{bmatrix}
        =
        \left(\begin{bmatrix}
        1&0  \\
        0&0        \end{bmatrix}
+\varepsilon \begin{bmatrix}
        0&0  \\
        1&1        \end{bmatrix}\right)\left(\begin{bmatrix}
        1   \\
        0         \end{bmatrix}
+\varepsilon \begin{bmatrix}
        0   \\
        1        \end{bmatrix}\right)\in \mathcal{R}(\widehat{M}).$$

  When  $\widehat{y}=\begin{bmatrix}
        0   \\
        1        \end{bmatrix}$,
then
$$\widehat{x}_2
 =
 \begin{bmatrix}
        1   \\
        0        \end{bmatrix}
+\varepsilon \begin{bmatrix}
        0  \\
        2        \end{bmatrix}
        =
        \left(\begin{bmatrix}
        1&0  \\
        0&0        \end{bmatrix}
+\varepsilon \begin{bmatrix}
        0&0  \\
        1&1        \end{bmatrix}\right)\left(\begin{bmatrix}
        1   \\
        0         \end{bmatrix}
+\varepsilon \begin{bmatrix}
        0   \\
        2        \end{bmatrix}\right)\in \mathcal{R}(\widehat{M})$$

        It is obvious that $\widehat{x}_1 \neq\widehat{x}_2 $.
Therefore,
the restricted  solution to  (\ref{DL-eQ-G-1}) in Theorem \ref{App-G-2} is not unique.
\end{example}

It is obvious  that
$\widehat{M}_{\rm W}^{\#}=\widehat{M}^{\#}$
and
$\widehat{M}-\left(\widehat{M}_{\rm W}^{\#}\right)^{\#}=O$,
when  ${ {\Dind}}\left(\widehat {M}\right)=1$.
Therefore, by applying Theorem \ref{App-G-2},
we have the following Corollary\ref{App-G-2-C}.

\begin{corollary} \label{App-G-2-C}
Let $\widehat M \in\mathbb{D}^{n\times n}$,
${ {\Dind}}\left(\widehat {M}\right)=1$,
$\widehat b \in\mathbb{D}^{n\times 1}$.
 Then the  solution to  (\ref{DL-eQ-G-1}) with
 $\widehat{x}\in \mathcal{R}(\widehat{M})$
 exists
 if and only if
$\left(I_n-\widehat{M}^{\#}\widehat{M}\right)\widehat{b}=O$.

Furthermore, when the restricted  solution to  (\ref{DL-eQ-G-1}) exists,
$ \widehat{x}=\widehat{M}^{\#}\widehat{b}$.
\end{corollary}

%

\section{Conclusion}\label{sec13}

In this paper,
we introduce a generalized DDI,
 call it WDDI,
 and get some properties and characterizations of the WDDI.
It should be noted that
any square dual matrix has the WDDI which is unique.
Furthermore,
we introduce the WDGI, which is a generalized DGI,
and
apply it to a constrained dual matrix problem.
 These results provide a strong theoretical basis for a wide range of applications of dual generalized inverse.

In \cite{Wang and Gao},
Wang and Gao get
$\widehat{M}^{\tiny\textcircled{\#}}
=\widehat{M}^{ \#}\widehat{M}\widehat{M}^{\dag}$,
when the DCI of $\widehat{M}$ exists.
Naturally, we denote $\widehat{M}_{\rm W}^{\tiny\textcircled{\#}}
=\widehat{M}_{\rm W}^{ \#}\widehat{M}\widehat{M}_{\rm W}^{\dag}$,
and
call it weak dual core inverse (WDCI),
when  the appreciable index of  $\widehat{M}\in\mathbb{D}^{n\times n} $ is 1.
And we obtain the characterizations of WDGI and WDGI
 by applying explicit dual matrix equations
in (\ref{g-DMPGI-Def}) and  (\ref{GDGI-def-Eq}), respectively.
Therefore, there seems to be an open problem that is both difficult and interesting:
How to provide the WDCI a similar characterization?

\bigskip

{\bf Funding}\quad
This work was supported partially by
the Guangxi Science and Technology Program (No. GUIKE AA24010005),
the Special Fund for Science and Technological Bases and Talents of Guangxi (No. GUIKE AD19245148)
and
the Research Fund Project of Guangxi Minzu University (No. 2019KJQD03).

{\bf Data availability}\quad
 All data generated or analysed during this study are included in this published article.

{\bf Conflict of interest}\quad
 The authors have no conflict of interest to declare that are relevant to the content of this
article.

\end{document}